\documentclass[12pt]{amsart}%
\usepackage{amsmath}
\usepackage{amsfonts}
\usepackage{amssymb}
\usepackage{amstext}
\usepackage{graphicx}%
\providecommand{\U}[1]{\protect \rule{.1in}{.1in}}
\providecommand{\U}[1]{\protect \rule{.1in}{.1in}}
\providecommand{\U}[1]{\protect \rule{.1in}{.1in}}
\newtheorem{theorem}{Theorem}[section]
\newtheorem{lemma}{Lemma}[section]

\newtheorem{corollary}{Corollary}[section]

\newtheorem{definition}{Definition}[section]

\numberwithin{equation}{section}

\theoremstyle{remark}
\newtheorem{remark}{Remark}[section]
\numberwithin{equation}{section}
\begin{document}
\title[Pseudo-Einstein, eigenvalue estimate and CR uniformization theorem ]{Pseudo-Einstein structure, eigenvalue estimate for the CR Paneitz operator and
its applications to uniformization theorem }
\author{Shu-Cheng Chang}
\address{Department of Mathematics and Taida Institute for Mathematical Sciences
(TIMS), National Taiwan University, Taipei 10617, Taiwan, R.O.C.}
\email{scchang@math.ntu.edu.tw }
\author{Ting-Jung Kuo}
\address{Department of Mathematics, National Taiwan Normal University, Taipei 11677,
Taiwan }
\email{tjkuo1215@ntnu.edu.tw}
\author{Chien Lin}
\address{Yau Mathematical Sciences Center, Tsinghua University, Haidian District,
Beijing 100084, China}
\email{chienlin@mail.tsinghua.edu.cn}

\begin{abstract}
In this note, we mainly focus on the existence of pseudo-Einstein contact
forms, an upper bound eigenvalue estimate for the CR Paneitz operator and its
applications to the uniformization theorem for Sasakian space form in an
embeddable closed strictly pseudoconvex CR $3$-manifold. Firstly, the
existence of pseudo-Einstein contact form is confirmed if the CR $3$-manifold
is Sasakian. Secondly, we derive an eigenvalue upper bound estimate for the CR
Paneitz operator and obtain the CR uniformization theorem for a class of CR
$3$-manifolds. At the end, under the positivity assumption of the
pseudohermitian curvature, we derive the existence theorem for pseudo-Einstein
contact forms and uniformization theorems in a closed strictly pseudoconvex CR
$3$-manifold of nonnegative CR Paneitz operator with kernel consisting of the
CR-pluriharmonic functions and the CR $Q$-curvature is CR-pluriharmonic.

\end{abstract}
\maketitle

\section{Introduction}

\ In Riemannian geometry, a Riemannian manifold is called Einstein if the
Ricci curvature tensor is function-proportional to its Riemannian metric. For
dimension greater than $2$, it is equivalent to the constant-proportional
case. In the contrast to the Riemannian geometry situation, there is a
resembling notion that a strictly pseudoconvex CR $(2n+1)$-manifold is called
pseudo-Einstein if the pseudohermitian Ricci curvature tensor is
function-proportional to its Levi metric. The pseudo-Einstein condition is
less rigid than the Einstein condition in Riemannian geometry. Indeed, the CR
contracted Bianchi identity no longer implies the pseudohermitian scalar
curvature $R$ to be a constant due to the presence of pseudohermitian torsion
for $n\geq2$
\[
R_{\alpha \overline{\beta}}{}_{,\beta}=R_{\alpha}-i(n-1)A_{\alpha \beta}%
{}_{,\overline{\beta}}.
\]
Note that any contact form on a closed strictly pseudoconvex $3$-manifold is
actually pseudo-Einstein since the pseudohermitian Ricci tensor has only one
component $R_{1\overline{1}}$.

In \cite{l}, J. Lee showed that the obstruction to the existence of a
pseudo-Einstein contact form $\theta$ is that its first Chern class
$c_{1}(T_{1,0}M)$ vanishes. Indeed, for a closed strictly pseudoconvex $(2n+1)
$-manifold $\left(  M,J,\theta \right)  $ with $c_{1}(T_{1,0}M)=0$ and $n\geq2
$, he proved that $M$ admits a globally defined pseudo-Einstein contact form
if either $M$ admits a contact form $\theta$ with nonnegative pseudohermitian
Ricci curvature tensor or the vanishing pseudohermitian torsion. However, his
method couldn't be applied to the case $n=1$ directly.

Then it is natural to focus on such existence theorem of pseudo-Einstein
contact forms for $n=1$. Of course, we must find another appropriate
definition for the pseudo-Einstein contact form. In fact by Lemma \ref{l1}
below, it is reasonable to view%
\[
W_{1}\doteqdot \left(  R,_{1}-iA_{11},_{\overline{1}}\right)  =0
\]
as the pseudo-Einstein contact form in a closed strictly pseudoconvex CR
$3$-manifold $\left(  M,J,\theta \right)  $.

Before we start to work on the existence of pseudo-Einstein contact forms in a
closed strictly pseudoconvex CR $3$-manifold, we make the following
observation in a closed strictly pseudoconvex CR $(2n+1)$-manifold $(M,J)$
with a choice of pseudohermitian contact form $\theta$.

(i) For $n\geq2$ : Assume that the pseudohermitian Ricci curvature is
positive, it is well-known (\cite{k}, \cite{l}) that we have the solvability
of the inhomogeneous tangential Cauchy-Riemann equation
\begin{subequations}
\begin{equation}
\overline{\partial}_{b}\varphi=\eta \label{2018B}%
\end{equation}
for any $\overline{\partial}_{b}$-closed $(0,1)$-form $\eta.$ That is to say
that
\end{subequations}
\[
H_{\overline{\partial}_{b}}^{0,1}(M)=0.
\]

(ii) For $n=1:$ We consider an embeddable closed strictly pseudoconvex CR
$3$-manifold $(M,\theta)$ with $c_{1}(T_{1,0}M)=0$. There is a pure imaginary
$1$-form%
\begin{equation}
\sigma=\sigma_{\overline{1}}\theta^{\overline{1}}-\sigma_{1}\theta^{1}%
+i\sigma_{0}\theta \label{25a}%
\end{equation}
such that
\[
d\omega_{1}^{1}=d\sigma.
\]
Due to J. J. Kohn's result (Lemma \ref{l3}), we observe that there are a
complex function
\[
\varphi=u+iv\in C_{%
%TCIMACRO{\U{2102} }%
%BeginExpansion
\mathbb{C}
%EndExpansion
}^{\infty}\left(  M\right)
\]
and $\gamma=\gamma_{\overline{1}}\theta^{\overline{1}}\in \Omega^{0,1}\left(
M\right)  \cap \ker \left(  \square_{b}\right)  $ such that%
\begin{equation}
\overline{\partial}_{b}\varphi=\sigma_{\overline{1}}\theta^{\overline{1}%
}-\gamma_{\overline{1}}\theta^{\overline{1}}\label{2018}%
\end{equation}
with
\[
\overline{\partial}_{b}(\sigma_{\overline{1}}\theta^{\overline{1}})=0.
\]
Here $\square_{b}=2\left(  \overline{\partial}_{b}\overline{\partial}%
_{b}^{\ast}+\overline{\partial}_{b}^{\ast}\overline{\partial}_{b}\right)  $ is
the Kohn-Rossi Laplacian. Then it is the natural question when we have the
solvability of the inhomogeneous tangential Cauchy-Riemann equation (i.e.
$\gamma=0$)
\begin{equation}
\overline{\partial}_{b}\varphi=\sigma_{\overline{1}}\theta^{\overline{1}%
}.\label{2018A}%
\end{equation}

In this paper, we mainly focus on the existence of pseudo-Einstein contact
forms as in Corollary \ref{c2}, Theorem \ref{t2}, Theorem \ref{t3a} and
Theorem \ref{t4}, an upper bound eigenvalue estimate for the CR Paneitz
operator as in Theorem \ref{t3}, Corollary \ref{c4} and its applications to
the uniformization theorem for Sasakian space form as in Corollary \ref{c4},
Corollary \ref{c5} and Corollary \ref{c6} in a closed strictly pseudoconvex CR
$3$-manifold.

We first state one of the main theorems as follows:

\begin{theorem}
\label{t1}If $(M,J,\theta)$ is an embeddable closed strictly pseudoconvex CR
$3$-manifold with $c_{1}(T_{1,0}M)=0$. \ Then

(i) $\widetilde{\theta}=e^{\frac{\left(  f+2u\right)  }{3}}\theta$ is a
pseudo-Einstein contact form if and only if $f$ satisfies the third-order
partial differential equation%
\begin{equation}
P_{1}f=i\left(  A_{11}\gamma_{\overline{1}}-\gamma_{1,0}\right)  .\label{7}%
\end{equation}
Here $P_{1}$ is a third-order CR-pluriharmonic operator
\[
P_{1}f=f_{\overline{1}11}+iA_{11}f_{\overline{1}}.
\]
(ii) In particular, $\widetilde{\theta}=e^{\frac{\left(  f+2u\right)  }{3}%
}\theta$ is a pseudo-Einstein contact form for a CR-pluriharmonic function $f$
if and only if the equality holds
\begin{equation}
\left(  A_{11}\gamma_{\overline{1}}-\gamma_{1,0}\right)  =0.\label{13}%
\end{equation}

\end{theorem}

In \cite{le}, Lempert showed that a closed strictly pseudoconvex CR
$3$-manifold $(M,J,\theta)$ with vanishing pseudohermitian torsion is
embeddable. As a consequence, we are able to show that one of existence
theorems for the pseudo-Einstein contact form in this paper.

\begin{corollary}
\label{c2} Let $(M,J,\theta)$ be a closed strictly pseudoconvex CR
$3$-manifold with $c_{1}(T_{1,0}M)=0$. Then $M$ admits a globally defined
pseudo-Einstein contact form $e^{\frac{\left(  f+2u\right)  }{3}}\theta$ for
any CR-pluriharmonic function $f$ if the pseudohermitian torsion is vanishing
(i.e. Sasakian). More precisely, we have
\[
\gamma_{1,0}=0.
\]

\end{corollary}

Note that we do not know whether it holds that $\gamma=0$ in the situation as
in Corollary \ref{c2}. However, by deriving the B\^{o}chner-type estimate as
in the Lemma \ref{l8}, we have%
\[%
%TCIMACRO{\dint \limits_{M}}%
%BeginExpansion
{\displaystyle \int \limits_{M}}
%EndExpansion
\left(  2R-Tor\right)  \left(  \gamma,\gamma \right)  d\mu+2%
%TCIMACRO{\dint \limits_{M}}%
%BeginExpansion
{\displaystyle \int \limits_{M}}
%EndExpansion
\left \vert \gamma_{1,1}\right \vert ^{2}d\mu+\frac{1}{2}%
%TCIMACRO{\dint \limits_{M}}%
%BeginExpansion
{\displaystyle \int \limits_{M}}
%EndExpansion
\left(  P_{0}f\right)  fd\mu=0
\]
if $\widetilde{\theta}=e^{\frac{\left(  f+2u\right)  }{3}}\theta$ is a
pseudo-Einstein contact form. It will conclude
\[
\gamma=0
\]
under certain pseudohermitian geometric assumptions and obtain the solvability
of the inhomogeneous tangential Cauchy-Riemann equation (\ref{2018A}).
Hereafter, we will use the result in \cite{ccy} implicitly, which says that if
$(M,J,\theta)$ is a closed strictly pseudoconvex CR $3$-manifold with
nonnegative CR Paneitz operator and positive Tanaka-Webster scalar curvature,
then $(M,J,\theta)$ is embeddable. This ensures the embeddability in the
hypothesis holds for closed strictly pseudoconvex CR $3$-manifolds.

\begin{theorem}
\label{t2} If $(M,J,\theta)$ is a closed strictly pseudoconvex CR $3$-manifold
with $c_{1}(T_{1,0}M)=0$ and nonnegative CR Paneitz operator $P_{0} $. Assume
that the pseudohermitian curvature is $\left(  \frac{1}{2},0\right)
$-positive
\[
R>|A_{11}|.
\]
Then $\widetilde{\theta}=e^{\frac{\left(  f+2u\right)  }{3}}\theta$ is a
pseudo-Einstein contact form for any CR-pluriharmonic function $f$ if and only
if the solvability of the inhomogeneous tangential Cauchy-Riemann equation
(\ref{2018A}).
\end{theorem}

We observe that, for a strictly pseudoconvex $3$-manifold $(M^{3},J,\theta)$,
we have the invariance property for the CR-pluriharmonic operator $P_{1}$ and
CR Paneitz operator $P_{0}$. It is to say that, for rescaled contact form
$\widetilde{\theta}=e^{2g}\theta,$ we have
\begin{equation}
\widetilde{P}_{1}=e^{-3g}P_{1}\text{ \textrm{and}\  \ }\widetilde{P}%
_{0}=e^{-4g}P_{0}.\label{A}%
\end{equation}
Then the nonnegativity of CR Paneitz operator $P_{0}$ is CR conformal
invariant (\cite{h}).

Since the CR Paneitz operator $P_{0}$ is nonnegative (\cite{ccc}) if the
pseudohermitian torsion is vanishing, it follows from Theorem \ref{t2} and
Corollary \ref{c2} that

\begin{corollary}
\label{c3} If $(M,J,\theta)$ is a closed strictly pseudoconvex CR $3$-manifold
with $c_{1}(T_{1,0}M)=0$. Assume that the manifold is Sasakian and the
Tanaka-Webster scalar curvature is positive, then we have the solvability of
the inhomogeneous tangential Cauchy-Riemann equation (\ref{2018A}). That is to
say that the Kohn--Rossi cohomology class of $\sigma_{\overline{1}}%
\theta^{\overline{1}}$ is vanishing.
\end{corollary}

Let $(M,J,\theta)$ be a closed strictly pseudoconvex CR $3$-manifold with
$c_{1}(T_{1,0}M)=0.$ With the notations as in section $2$ and section $3,$
another B\^{o}chner-type equality holds%
\[
\int_{M}\left(  R-\frac{1}{2}Tor-\frac{1}{2}Tor^{\prime}\right)  \left(
\gamma,\gamma \right)  d\mu+\int_{M}\left \vert \gamma_{1,1}\right \vert ^{2}%
d\mu+\int_{M}Qud\mu+\int_{M}\left(  P_{0}u\right)  ud\mu=0.
\]
With the help of the notion of $\left(  C_{0},C_{1}\right)  $%
\textit{-convexity}, we have the eigenvalue estimate for the CR Paneitz
operator $P_{0}$ in terms of the CR \ $Q$-curvature\textit{. }

\begin{theorem}
\label{t3} Let $(M,J,\theta)$ be a closed strictly pseudoconvex CR
$3$-manifold of $c_{1}(T_{1,0}M)=0$ and nonnegative CR Paneitz operator
$P_{0}$ with kernel consisting of the CR-pluriharmonic functions. Assume that
\ the pseudohermitian curvature is $\left(  \frac{1}{2},\frac{1}{2}\right)
$-positive
\[
R>(|A_{11}|+\left \vert A_{11,\overline{1}}\right \vert ).
\]
If $\widetilde{\theta}=e^{\frac{\left(  f+2u\right)  }{3}}\theta$ is a
pseudo-Einstein contact form for any CR-pluriharmonic function $f$, then one
can derive the upper bound estimate for the first eigenvalue of the CR Paneitz
operator $P_{0}$
\begin{equation}
\Lambda^{2}\int_{M}(u^{\bot})^{2}d\mu \leq \int_{M}(Q^{\bot})^{2}d\mu
\label{2018H}%
\end{equation}
with the decomposition $Q=Q_{\ker}+Q^{\perp}$ and $u=u_{\ker}+u^{\perp}$. Here
$\Lambda$ is the positive constant as in (\ref{41}).
\end{theorem}

\begin{remark}
1. (\cite{h}) For a closed strictly pseudoconvex CR $3$-manifold of vanishing
pseudohermitian torsion (Sasakian), we have
\begin{equation}
\ker P_{1}=\ker P_{0}.\label{2018C}%
\end{equation}
Furthermore, the real ellipsoids in $\mathbf{C}^{2}$ are such that the CR
Paneitz operators are nonnegative with kernel consisting of the
CR-pluriharmonic functions (\cite{ccya}). In general for non-embeddable CR $3
$-manifolds, we only have
\[
\ker P_{1}\varsubsetneq \ker P_{0}.
\]

2. As in the Remark \ref{r2}, $M$ admits a Riemannian metric of positive
scalar curvature if the pseudohermitian curvature is $\left(  \frac{1}%
{2},\frac{1}{2}\right)  $-positive. Then the properties of pseudohermitian
curvature positivity and nonnegative CR Paneitz operator $P_{0}$ imply the
embeddability of $(M,J,\theta_{0})$\ in the complex Euclidean space
$\mathbf{C}^{N}$ (\cite{ccy}). Furthermore, the Paneitz operator $P_{0}$ with
respect to $(J,\theta)$ is essentially positive in this special case as well.
\end{remark}

Combining Theorem \ref{t3}, Corollary \ref{c3}, (\ref{0}) and the above
remark, we have the following CR uniformization theorem (\cite{t}) in a
Sasakian manifold due to the eigenvalue estimate of the CR Paneitz operator
(\ref{2018H}).

\begin{corollary}
\label{c4} If $(M,J,\theta)$ is a closed strictly pseudoconvex CR $3$-manifold
with $c_{1}(T_{1,0}M)=0$. Assume that the manifold is Sasakian and the
Tanaka-Webster scalar curvature is positive, then
\[
\Lambda^{2}\int_{M}(u^{\bot})^{2}d\mu \leq \int_{M}(Q^{\bot})^{2}d\mu.
\]
In additional, if the CR \ $Q$-curvature is pluriharmonic (i.e. $Q^{\bot}=0),$
then $(M,J,\theta)$ is the Sasakian space form with the positive constant
Tanaka-Webster scalar curvature and vanishing torsion.
\end{corollary}

Finally, if we do not assume the torsion is vanishing (non-Sasakian), we can
derive another existence theorem for the pseudo-Einstein contact form with the
stronger condition.

\begin{theorem}
\label{t3a} Let $(M,J,\theta)$ be a closed strictly pseudoconvex CR
$3$-manifold of $c_{1}(T_{1,0}M)=0$ and nonnegative CR Paneitz operator
$P_{0}$ with kernel consisting of the CR-pluriharmonic functions. Assume that
\ the pseudohermitian curvature is $\left(  \frac{1}{2},\frac{1}{2}\right)
$-positive
\[
R>(|A_{11}|+\left \vert A_{11,\overline{1}}\right \vert ).
\]
If the CR \ $Q$-curvature is pluriharmonic, then $\widetilde{\theta}%
=e^{\frac{\left(  f+2u\right)  }{3}}\theta$ is a pseudo-Einstein contact form
for any CR-pluriharmonic function $f$ with
\[
\gamma=0.
\]

\end{theorem}

As a consequence, we have the CR uniformization theorem (\cite{t}) in a
spherical CR $3$-manifold.

\begin{corollary}
\label{c5} Let $(M,J,\theta)$ be a closed strictly pseudoconvex CR
$3$-manifold of $c_{1}(T_{1,0}M)=0$ and nonnegative CR Paneitz operator
$P_{0}$ with kernel consisting of the CR-pluriharmonic functions. Assume that
\ the pseudohermitian curvature is $\left(  \frac{1}{2},\frac{1}{2}\right)
$-positive and the CR \ $Q$-curvature is pluriharmonic, then $(M,J,\theta)$ is
a pseudo-Einstein CR manifold. In additional, if it is spherical with positive
constant Tanaka-Webster scalar curvature, then $(M,J,\theta)$ is the Sasakian
space form with the positive constant Tanaka-Webster scalar curvature and
vanishing torsion.
\end{corollary}

In the course of the proof of Theorem \ref{t3a}, we assume that the CR
\ $Q$-curvature is pluriharmonic. However, in particular for assuming
vanishing of the CR $Q$-curvature \textit{(see section }$2$\textit{)}, we can
trop the condition of (\ref{2018C}) as the following :

\begin{theorem}
\label{t4} Let $(M,J,\theta)$ is a closed strictly pseudoconvex CR
$3$-manifold of $c_{1}(T_{1,0}M)=0$ with nonnegative CR Paneitz operator
$P_{0}. $ Assume that \ the pseudohermitian curvature is $\left(  \frac{1}%
{2},\frac{1}{2}\right)  $-positive
\[
R>(|A_{11}|+\left \vert A_{11,\overline{1}}\right \vert )
\]
and the CR $Q$-curvature is vanishing
\[
\Delta_{b}R-i(A_{11,\bar{1}\bar{1}}-A_{\bar{1}\bar{1},11})=0.
\]
Then $\widetilde{\theta}=e^{\frac{\left(  f+2u\right)  }{3}}\theta$ is a
pseudo-Einstein contact form for any CR-pluriharmonic function $f.$
\end{theorem}

\begin{corollary}
\label{c6} Let $(M,J,\theta)$ is a closed strictly pseudoconvex CR
$3$-manifold of $c_{1}(T_{1,0}M)=0$ with nonnegative CR Paneitz operator
$P_{0}. $ Assume that \ the pseudohermitian curvature is $\left(  \frac{1}%
{2},\frac{1}{2}\right)  $-positive and the CR $Q$-curvature is vanishing. Then
$(M,J,\theta)$ is a pseudo-Einstein CR manifold. In additional if it is
spherical with positive constant Tanaka-Webster scalar curvature, then
$(M,J,\theta)$ is the Sasakian space form with the positive constant
Tanaka-Webster scalar curvature and vanishing torsion.
\end{corollary}

We conclude this introduction with a brief plan of the paper. In Section~$2$,
we derive some preliminary results and indicate the geometry and topology of
CR $3$-manifolds with the positivity of pseudohermitian curvature. In Section
$3,\ $we prove main Theorems. In Appendix, we survey basic notions in the
pseudohermitian (strictly pseudoconvex CR) geometry.

\noindent \textbf{Acknowledgements} Part of the project was done during the
visit to Yau Mathematical Sciences Center, Tsinghua University. The authors
would like to express their thanks for the warm hospitality.

\section{Preliminaries}

In this section, we derive some necessary ingredients for the proof of main
results in this paper. In particular, we define the positivity of
pseudohermitian curvature and indicate the geometry and topology of strictly
pseudoconvex CR $3$-manifolds. Let $C_{0},C_{1}$ be both nonnegative numbers.

\begin{definition}
We say that a strictly pseudoconvex CR $3$-manifolds is $\left(  C_{0}%
,C_{1}\right)  $\textit{-convex} if the pseudohermitian curvature is $\left(
C_{0},C_{1}\right)  $-positive. That is%
\begin{equation}%
\begin{array}
[c]{ccl}%
(R-C_{0}Tor-C_{1}Tor^{\prime})(X,X) & = & Rx^{1}x^{\bar{1}}-2C_{0}%
\mathrm{Re}[i(A_{\bar{1}\bar{1}}x^{\bar{1}}x^{\bar{1}})]\\
&  & -2C_{1}\mathrm{Re[}(iA_{\bar{1}\bar{1},1}x^{\bar{1}})]>0,
\end{array}
\label{33}%
\end{equation}
$\mathrm{for\ any}\ X=x^{1}Z_{1}\in T_{1,0}(M).$
\end{definition}

Before giving the proof of Theorem \ref{t3}, we explain why we introduce the
notion of $\left(  C_{0},C_{1}\right)  $-convexity as follows:

\begin{lemma}
\label{l6}Let $M$ be a closed strictly pseudoconvex CR $3$-manifold. For any
nonnegative constant $C_{0},C_{1}$; $\left(  C_{0},C_{1}\right)  $-convexity
is equivalent to the curvature-torsion pinching condition
\[
R(x)>2\left(  C_{0}|A_{11}|+C_{1}\left \vert A_{11,\overline{1}}\right \vert
\right)  (x)
\]
for all $x\in M$.
\end{lemma}

\begin{remark}
\label{r2} (\cite{cacc}) Let $(M^{3},J,\theta)$ be a closed strictly
pseudoconvex CR $3$-manifold with
\[
R(x)>\left \vert A_{11}(x)\right \vert
\]
which is $\left(  C_{0},0\right)  $-positive. Then $M$ admits a Riemannian
metric of positive scalar curvature.
\end{remark}

\begin{proof}
Fix a point $x\in M$ and denote
\[
\left \{
\begin{array}
[c]{l}%
A_{11}(x)=a(x)+b(x)i\\
A_{11,\overline{1}}(x)=c(x)+d(x)i
\end{array}
\right.  .
\]
Without loss of generality, one may consider $X=x^{1}Z_{1}=(1+si)Z_{1}$ for
some $s\in \mathbb{R}$. The convexity condition $\left(  \ref{33}\right)  $
reads as
\[%
\begin{array}
[c]{ccl}%
R(1+si)(1-si) & > & i[C_{0}(a-bi)(1-si)^{2}+C_{1}\left(  c-di\right)
(1-si)-conj.]\\
& = & -2\left[  \left(  C_{0}b\right)  s^{2}-\left(  2C_{0}a+C_{1}c\right)
s+\left(  -C_{1}d-C_{0}b\right)  \right]  .
\end{array}
\]
i.e.,
\begin{equation}
R>f\left(  s\right)  :=-2\left[  \left(  C_{0}b\right)  -\frac{\left(
2C_{0}b+C_{1}d\right)  +\left(  2C_{0}a+C_{1}c\right)  s}{1+s^{2}}\right]
.\label{41b}%
\end{equation}
Since $X$ is arbitrary, this inequality holds for all $s\in \mathbb{R}$. We
have
\begin{equation}%
\begin{array}
[c]{ccl}%
f(s) & \leq & \max_{s\in \mathbb{R}}f(s)\\
& = & f(s_{0})\\
& = & 2\sqrt{\left(  C_{0}a+C_{1}\frac{c}{2}\right)  ^{2}+\left(  C_{0}%
b+C_{1}\frac{d}{2}\right)  ^{2}}+C_{1}d\\
& \leq & 2\left(  C_{0}|A_{11}|+C_{1}\left \vert A_{11,\overline{1}}\right \vert
\right)
\end{array}
\label{41a}%
\end{equation}
where
\[
s_{0}=\frac{-\left(  C_{0}b+C_{1}\frac{d}{2}\right)  +\sqrt{\left(
C_{0}a+C_{1}\frac{c}{2}\right)  ^{2}+\left(  C_{0}b+C_{1}\frac{d}{2}\right)
^{2}}}{\left(  C_{0}a+C_{1}\frac{c}{2}\right)  },
\]
for $C_{0}a+C_{1}\frac{c}{2}\neq0$, is a critical number of $f$. Thus
\[
R(x)>2\left(  C_{0}|A_{11}|+C_{1}\left \vert A_{11,\overline{1}}\right \vert
\right)  (x)
\]
for all $x\in M$. As for the case of $C_{0}a+C_{1}\frac{c}{2}=0$, the same
deduction could be applied. Therefore, $\left(  C_{0},C_{1}\right)
$-positivity is equivalent to
\[
R>2\left(  C_{0}|A_{11}|+C_{1}\left \vert A_{11,\overline{1}}\right \vert
\right)  .
\]

\end{proof}

\begin{definition}
(\cite{l}) (i) A contact form $\theta$ on a closed strictly pseudoconvex CR
$(2n+1)$-manifold $(M,\theta)$ is said to be pseudo-Einstein for $n\geq2$ if
the pseudohermitian Ricci tensor $R_{\alpha \overline{\beta}}$ is proportional
to the Levi form $h_{\alpha \overline{\beta}}$, i.e.,
\[
\ R_{\alpha \overline{\beta}}=\frac{R}{n}h_{\alpha \overline{\beta}},
\]
where $R=h^{\alpha \overline{\beta}}R_{\alpha \overline{\beta}}$ is the
Tanaka-Webster scalar curvature of $(J,\theta).$

(ii) (Lemma \ref{l1}) Note that any contact form on a closed strictly
pseudoconvex $3$-manifold is actually pseudo-Einstein (since the
pseudohermitian Ricci tensor has only one component $R_{1\overline{1}}$). Then
we define a contact form $\theta$ on a closed strictly pseudoconvex CR
$3$-manifold $(M,\theta)$ is said to be pseudo-Einstein if the following
tensor is vanishing
\[
W_{1}\doteqdot \left(  R_{1}-iA_{11_{,}\overline{1}}\right)  =0.
\]
(iii) We define the first Chern class $c_{1}(T_{1,0}M)\in H^{2}(M,\mathbf{R})
$ for the holomorphic tangent bundle $T^{1,0}M$ by
\begin{align*}
c_{1}(T^{1,0}M)  & =\frac{i}{2\pi}[d\omega_{\alpha}{}^{\alpha}]\\
& =\frac{i}{2\pi}[R_{\alpha \overline{\beta}}\theta^{\alpha}\wedge
\theta^{\overline{\beta}}+A_{\alpha \mu,\overline{\alpha}}{}\theta^{\mu}%
\wedge \theta-A_{\overline{\alpha}\overline{\mu},\alpha}\theta^{\overline{\mu}%
}\wedge \theta].
\end{align*}
(iv) Note that any pseudo-Einstein manifold $(M^{2n+1},\theta)$, the first
Chern class $c_{1}(T_{1,0}M)$ of $T_{1,0}(M)$ is vanishing (\cite{l}).
\end{definition}

Next let us recall the equivalent definitions of the pseudo-Einsteinian
$(2n+1)$-manifold for $n\geq2$ and $n=1$ as well.

\begin{lemma}
\label{l1} (i) If $(M,J,\theta)$ is a strictly pseudoconvex CR $(2n+1)$%
-manifold for $n\geq2$, then the following propositions are all equivalent :%
\[%
\begin{array}
[c]{cl}%
\left(  1\right)  & R_{\alpha \overline{\beta}}=\frac{R}{n}h_{\alpha
\overline{\beta}},\text{ \  \ }\\
\left(  2\right)  & \left(  \omega_{\alpha}^{\alpha}+\frac{i}{n}%
R\theta \right)  \text{ }\mathrm{is}\text{ }\mathrm{closed,}\\
\left(  3\right)  & W_{\alpha}\doteqdot \left(  R,_{\alpha}-inA_{\alpha \beta
},^{\beta}\right)  =0\text{.\ }%
\end{array}
\]
As for $n=1,$ we still have the equivalent between $(2)$ and $(3).$

(ii) By the equivalence of $\left(  2\right)  $ and $\left(  3\right)  $, we
see the first Chern class $c_{1}(T_{1,0}M)$ is vanishing if $(M,J,\theta)$ is
a pseudo-Einsteinian $3$-manifold.
\end{lemma}

\begin{proof}
The equivalence of $\left(  1\right)  $ and $\left(  2\right)  $ could be
found in \cite{l} for $n\geq2$. The proof of $\left(  2\right)
\Longleftrightarrow \left(  3\right)  $ for $n\geq2$ is the same with $n=1$.
So, for simplification, we just give the proof of the equivalence of $\left(
2\right)  \Longleftrightarrow \left(  3\right)  $ for $n=1$.

Because
\[
d\omega_{1}^{1}=R\theta^{1}\wedge \theta^{\overline{1}}+A_{11,\overline{1}%
}\theta^{1}\wedge \theta-A_{\overline{1}\overline{1},1}\theta^{\overline{1}%
}\wedge \theta,
\]
we have%
\[%
\begin{array}
[c]{ccl}%
d\left(  \omega_{1}^{1}+iR\theta \right)  & = & d\omega_{1}^{1}+i\left(
R_{,1}\theta^{1}+R_{,\overline{1}}\theta^{\overline{1}}\right)  \wedge
\theta-R\theta^{1}\wedge \theta^{\overline{1}}\\
& = & i\left[  \left(  R_{,1}-iA_{11,\overline{1}}\right)  \theta^{1}+\left(
R_{,\overline{1}}+iA_{\overline{1}\overline{1},1}\right)  \theta^{\overline
{1}}\right]  \wedge \theta
\end{array}
.
\]

Hence%

\[
d\left(  \omega_{1}^{1}+iR\theta \right)  =0\Longleftrightarrow R_{,1}%
-iA_{11,\overline{1}}=0.
\]

\end{proof}

We recall some useful notations as well.

\begin{definition}
(\cite{l}) (i) Let $(M,J,\theta)$ be a three-dimensional strictly pseudoconvex
CR manifold. We define
\[
P\varphi=(P_{1}\varphi)\theta^{1},
\]
which is an operator that characterizes CR-pluriharmonic functions. Here
$P_{1}\varphi=\varphi_{\bar{1}}{}^{\bar{1}}{}_{1}+iA_{11}\varphi^{1}$ and
$\overline{P}\varphi=(\overline{P}_{1})\theta^{\bar{1}}$, the conjugate of
$P$. The CR Paneitz operator $P_{0}$ is defined by%
\begin{equation}
P_{0}\varphi=\left(  \delta_{b}(P\varphi)+\overline{\delta}_{b}(\overline
{P}\varphi)\right)  ,\label{id9}%
\end{equation}
where $\delta_{b}$ is the divergence operator that takes $(1,0)$-forms to
functions by $\delta_{b}(\sigma_{1}\theta^{1})=\sigma_{1,}{}^{1}$ and,
similarly, $\bar{\delta}_{b}(\sigma_{\bar{1}}\theta^{\bar{1}})=\sigma_{\bar
{1},}{}^{\bar{1}}$. We observe that%
\begin{equation}
\int_{M}\langle P\varphi+\overline{P}\varphi,d_{b}\varphi \rangle_{L_{\theta
}^{\ast}}\ d\mu=-\int_{M}P_{0}\varphi \cdot \varphi \ d\mu \label{10}%
\end{equation}
with $d\mu=\theta \wedge d\theta.$ One can check that $P_{0}$ is self-adjoint,
that is, $\left \langle P_{0}\varphi,\psi \right \rangle =\left \langle
\varphi,P_{0}\psi \right \rangle $ for all smooth functions $\varphi$ and $\psi
$. For the details about these operators, the reader can make reference to
\cite{gl}, \cite{h}, \cite{l}, \cite{gg} and \cite{fh}.

(ii) On a complete pseudohermitian $3$-manifold $(M,J,\theta),$ we call the
Paneitz operator $P_{0}$ with respect to $(J,\theta)$ essentially positive if
there exists a constant $\Lambda$ $>$ $0$ such that
\begin{equation}
\int_{M}P_{0}\varphi \cdot \varphi d\mu \geq \Lambda \int_{M}\varphi^{2}%
d\mu.\label{41}%
\end{equation}
for all real smooth functions $\varphi$ $\in(\ker P_{0})^{\perp}$ (i.e.
perpendicular to the kernel of $P_{0}$ in the $L^{2}$ norm with respect to the
volume form $d\mu$ $=$ $\theta \wedge d\theta).$ We say that $P_{0}$ is
nonnegative if
\[
\int_{M}P_{0}\varphi \cdot \varphi d\mu \geq0
\]
for all real smooth functions $\varphi$.
\end{definition}

\begin{remark}
\label{r1} 1. The notions of Paneitz operator $P_{0}$ and $Q$-curvature were
initially introduced on a Riemannian manifold, and were considered as a kind
of generalization of Laplacian and Gaussian curvature on a two-dimensional
manifold, respectively (\cite{h}).

2. The kernel of the CR Paneitz operator $P_{0}$ is infinite dimensional,
containing all CR-pluriharmonic functions.

3. Let $(M,J,\theta)$ be a closed strictly pseudoconvex $3$-manifold with
vanishing pseudohermitian torsion. Then the corresponding CR Paneitz operator
$P_{0}$ is essentially positive ( \cite{ccc}, \cite{cac}).
\end{remark}

Finally, we define the CR $Q$-curvature in a pseudohermitian $3$-manifold by
\begin{equation}
Q:=-\mathrm{Re}(R_{_{,}1}-iA_{11,\bar{1}})_{\bar{1}}=-\mathrm{Re}(R_{,1\bar
{1}}-iA_{11,\bar{1}\bar{1}}).\label{c}%
\end{equation}
Then
\[
Q=-\frac{1}{2}[\Delta_{b}R-i(A_{11,\bar{1}\bar{1}}-A_{\bar{1}\bar{1},11})].
\]
Now for $\theta=e^{2\gamma}\theta_{0}$, under this conformal change, it is
known that we have the following transformation laws (\cite{h}) :
\begin{equation}
Q=e^{-4\gamma}(Q_{0}+\frac{3}{4}P_{0}\gamma)\label{0b}%
\end{equation}
and
\begin{equation}
W_{1}:=(R_{,1}-iA_{11},_{\bar{1}})=e^{-3\gamma}[\overset{0}{R}_{,1}%
-i\overset{0}{A}_{11,\bar{1}}-6\overset{0}{P}_{1}\gamma],\label{0c}%
\end{equation}
where $P_{0}$ and $Q_{0}$ denote the CR Paneitz operator and the CR
$Q$-curvature with respect to $(M,J,$ $\theta_{0})$, respectively.

Finally, we recall that

\begin{definition}
We call a CR structure $J$ spherical if Cartan curvature tensor $Q_{11}$
vanishes identically. Here
\[
Q_{11}=\frac{1}{6}R_{11}+\frac{i}{2}RA_{11}-A_{11,0}-\frac{2i}{3}%
A_{11,\overset{\_}{1}1}.
\]
Note that $(M,J,\theta)$ is called a spherical CR $3$-manifold if $J$ is a
spherical structure. We observe that the spherical structure is CR invariant
and a closed spherical CR $3$-manifold $(M,J,\theta)$ is locally CR equivalent
to the standard CR $3$-sphere $(\mathbf{S}^{3},\widehat{J},\widehat{\theta}).$
In additional, if $M$ is simply connected, then $(M,J,\theta)$ is the standard
CR $3$-sphere.
\end{definition}

\section{Proofs of Main Theorems}

In this section, we prove the main theorems. We start from the groundwork for
Theorem \ref{t1}.

\begin{lemma}
\label{l2} If $(M,J,\theta)$ is a strictly pseudoconvex $3$-manifold with
$c_{1}(T_{1,0}M)=0$, then there is a pure imaginary $1$-form
\[
\sigma=\sigma_{\overline{1}}\theta^{\overline{1}}-\sigma_{1}\theta^{1}%
+i\sigma_{0}\theta
\]
with $d\omega_{1}^{1}=d\sigma$ such that%
\begin{equation}
\left \{
\begin{array}
[c]{l}%
R=R_{1\overline{1}}=\sigma_{\overline{1},1}+\sigma_{1,\overline{1}}-\sigma
_{0}\\
A_{11,\overline{1}}=\sigma_{1,0}+i\sigma_{0,1}-A_{11}\sigma_{\overline{1}}%
\end{array}
\right.  .\label{1}%
\end{equation}

\end{lemma}

\begin{proof}
Because%
\[
c_{1}(T_{1,0}M)=-\frac{1}{2\pi i}\left[  d\omega_{1}^{1}\right]  =0,
\]

we know there is a pure imaginary $1$-form%
\[
\sigma=\sigma_{\overline{1}}\theta^{\overline{1}}-\sigma_{1}\theta^{1}%
+i\sigma_{0}\theta
\]
such that%
\[
d\omega_{1}^{1}=d\sigma.
\]

By the structure equation%
\[
\left \{
\begin{array}
[c]{l}%
d\theta=i\theta^{1}\wedge \theta^{\overline{1}}\\
d\theta^{1}=A_{\overline{1}\overline{1}}\theta \wedge \theta^{1}%
\end{array}
\right.  ,
\]
we have%
\[%
\begin{array}
[c]{ccl}%
d\sigma & = & \left(  \sigma_{\overline{1},1}\theta^{1}+\sigma_{\overline
{1},0}\theta \right)  \wedge \theta^{\overline{1}}+\sigma_{\overline{1}}%
d\theta^{\overline{1}}-\left(  \sigma_{1,\overline{1}}\theta^{\overline{1}%
}+\sigma_{1,0}\theta \right)  \wedge \theta^{1}-\\
&  & \sigma_{1}d\theta^{1}+i\left(  \sigma_{0,1}\theta^{1}+\sigma
_{0,\overline{1}}\theta^{\overline{1}}\right)  \wedge \theta+i\sigma_{0}%
d\theta \\
& = & \left(  \sigma_{\overline{1},1}+\sigma_{1,\overline{1}}-\sigma
_{0}\right)  \theta^{1}\wedge \theta^{\overline{1}}-\left(  \sigma
_{1,0}+i\sigma_{0,1}-\sigma_{\overline{1}}A_{11}\right)  \theta \wedge
\theta^{1}+\\
&  & \left(  \sigma_{\overline{1},0}-i\sigma_{0,\overline{1}}-\sigma
_{1}A_{\overline{1}\overline{1}}\right)  \theta \wedge \theta \overline{^{1}}.
\end{array}
.
\]

Due to
\[
d\sigma=d\omega_{1}^{1}=R_{1\overline{1}}\theta^{1}\wedge \theta^{\overline{1}%
}+A_{11,\overline{1}}\theta^{1}\wedge \theta-A_{\overline{1}\overline{1}%
,1}\theta^{\overline{1}}\wedge \theta,
\]
we derive%
\[
\left \{
\begin{array}
[c]{l}%
R=R_{1\overline{1}}=\sigma_{\overline{1},1}+\sigma_{1,\overline{1}}-\sigma
_{0}\\
A_{11,\overline{1}}=\sigma_{1,0}+i\sigma_{0,1}-A_{11}\sigma_{\overline{1}}%
\end{array}
\right.  .
\]

\end{proof}

\bigskip

We would need the J.J. Kohn's Hodge theory for the $\overline{\partial}_{b}$
complex (see \cite{k}) :

\begin{lemma}
\label{l3}If $(M,J,\theta)$ is an embeddable closed strictly pseudoconvex CR
$3$-manifold and $\eta \in \Omega^{0,1}\left(  M\right)  $, a smooth $\left(
0,1\right)  $-form on $M$ with
\[
\overline{\partial}_{b}\eta=0,
\]
then there are a complex function $\varphi \in C_{%
%TCIMACRO{\U{2102} }%
%BeginExpansion
\mathbb{C}
%EndExpansion
}^{\infty}\left(  M\right)  $ and $\gamma \in \Omega^{0,1}\left(  M\right)  $
such that
\[
\left(  \eta-\overline{\partial}_{b}\varphi \right)  =\gamma \in \ker \left(
\square_{b}\right)  ,
\]
where $\square_{b}=2\left(  \overline{\partial}_{b}\overline{\partial}%
_{b}^{\ast}+\overline{\partial}_{b}^{\ast}\overline{\partial}_{b}\right)  $ is
the Kohn-Rossi Laplacian.
\end{lemma}

Subsequently, we deduce the expression for $W_{1}$. We denote $\gamma
_{1}:=\overline{\gamma_{\overline{1}}}.$

\begin{lemma}
\label{l4} If $(M,J,\theta)$ is an embeddable closed strictly pseudoconvex CR
$3$-manifold with $c_{1}(T_{1,0}M)=0$, then there are $u\in C_{%
%TCIMACRO{\U{211d} }%
%BeginExpansion
\mathbb{R}
%EndExpansion
}^{\infty}\left(  M\right)  $ and $\gamma=\gamma_{\overline{1}}\theta
^{\overline{1}}\in \Omega^{0,1}\left(  M\right)  $ with
\[
\gamma_{\overline{1},1}=\gamma_{1,\overline{1}}=0
\]
such that
\begin{equation}
W_{1}=2P_{1}u+i\left(  A_{11}\gamma_{\overline{1}}-\gamma_{1,0}\right)
.\label{0}%
\end{equation}

\end{lemma}

\begin{proof}
By choosing $\eta=\sigma_{\overline{1}}\theta^{\overline{1}}$ as in Lemma
\ref{l3}, where $\sigma$ is chosen from Lemma \ref{l2},\textbf{\ }there are%
\[
\varphi=u+iv\in C_{%
%TCIMACRO{\U{2102} }%
%BeginExpansion
\mathbb{C}
%EndExpansion
}^{\infty}\left(  M\right)
\]
and%
\[
\gamma=\gamma_{\overline{1}}\theta^{\overline{1}}\in \Omega^{0,1}\left(
M\right)  \cap \ker \left(  \square_{b}\right)
\]
such that%
\begin{equation}
\sigma_{\overline{1}}=\varphi_{\overline{1}}+\gamma_{\overline{1}}\label{18}%
\end{equation}

Note that%
\begin{equation}
\square_{b}\gamma=0\Longrightarrow \overline{\partial}_{b}^{\ast}%
\gamma=0\Longrightarrow \gamma_{\overline{1},1}=0\label{2}%
\end{equation}
and%
\begin{equation}
\sigma_{1}=\left(  \overline{\varphi}\right)  _{1}+\gamma_{1}.\label{19}%
\end{equation}

Thus%
\[%
\begin{array}
[c]{ccl}%
\sigma_{1,\overline{1}1} & = & (\overline{\varphi})_{,1\overline{1}1}%
+\gamma_{1,\overline{1}1}\text{ \ }\mathrm{by}\text{ }\left(  \ref{19}\right)
\\
& = & (\overline{\varphi})_{,1\overline{1}1}\text{ \ }\mathrm{by}\text{
}\left(  \ref{2}\right) \\
& = & (\overline{\varphi})_{,\overline{1}11}+i(\overline{\varphi})_{,01}\text{
\ }\mathrm{by}\text{ }\left(  \ref{6}\right) \\
& = & (\overline{\varphi})_{,\overline{1}11}+i\left[  (\overline{\varphi
})_{,10}+A_{11}(\overline{\varphi})_{,\overline{1}}\right]  \text{
\ }\mathrm{by}\text{ }\left(  \ref{6}\right)
\end{array}
\]
and%
\[
\sigma_{\overline{1},11}=\varphi_{,\overline{1}11}\text{ \ }\mathrm{from}%
\text{ }\left(  \ref{18}\right)  \text{ }\mathrm{and}\text{ }\left(
\ref{2}\right)
\]
imply%
\[%
\begin{array}
[c]{ccl}%
W_{1} & = & R_{,1}-iA_{11,\overline{1}}\\
& = & \sigma_{\overline{1},11}+\sigma_{1,\overline{1}1}-i\sigma_{1,0}%
+iA_{11}\sigma_{\overline{1}}\text{ \ }\mathrm{by}\text{ }\left(
\ref{1}\right) \\
& = & \varphi_{,\overline{1}11}+(\overline{\varphi})_{,\overline{1}11}%
+iA_{11}(\overline{\varphi})_{,\overline{1}}-i\gamma_{1,0}+iA_{11}\left(
\varphi_{\overline{1}}+\gamma_{\overline{1}}\right) \\
& = & 2\left(  u_{\overline{,1}11}+iA_{11}u_{\overline{1}}\right)  +i\left(
A_{11}\gamma_{\overline{1}}-\gamma_{1,0}\right) \\
& = & 2P_{1}u+i\left(  A_{11}\gamma_{\overline{1}}-\gamma_{1,0}\right)
\end{array}
.
\]

This completes the proof.
\end{proof}

Now we are ready to give the proof of Theorem \ref{t1}:

\begin{proof}
(Proof of Theorem \ref{t1}\textbf{) }Set%
\[
\widetilde{\theta}=e^{2\lambda}\theta \text{.}%
\]
By the transformation law (refer to Lemma 5.4 in \cite{h} or Lemma 3.1 in
\cite{cw}), we know%
\begin{equation}
\widetilde{W_{1}}=e^{-3\lambda}\left(  W_{1}-6P_{1}\lambda \right)  ,\label{5}%
\end{equation}
where the notation with "tilde" means such quantity corresponds to the new
contact form $\widetilde{\theta}$. With the help of $\left(  \ref{5}\right)  $
and Lemma \ref{l4}, we have%
\[
\widetilde{W_{1}}=0
\]
if and only if%
\[
W_{1}=6P_{1}\lambda
\]
if and only if
\[
6P_{1}\lambda=2P_{1}u+i\left(  A_{11}\gamma_{\overline{1}}-\gamma
_{1,0}\right)  .
\]
That is to say
\[
P_{1}f=i\left(  A_{11}\gamma_{\overline{1}}-\gamma_{1,0}\right)
\]
for
\[
f=\left(  6\lambda-2u\right)  .
\]

\end{proof}

\begin{remark}
From $\left(  \ref{7}\right)  $ and
\[
\gamma_{1,0\overline{1}}=\gamma_{1,\overline{1}0}+A_{\overline{1}\overline{1}%
}\gamma_{1,1}+A_{\overline{1}\overline{1},1}\gamma_{1}=A_{\overline
{1}\overline{1}}\gamma_{1,1}+A_{\overline{1}\overline{1},1}\gamma_{1},
\]
we could deduce $f$ satisfies the fourth-order partial differential equation
\begin{equation}
P_{0}f=2i\left[  \left(  A_{11}\gamma_{\overline{1}}\right)  _{,\overline{1}%
}-\left(  A_{\overline{1}\overline{1}}\gamma_{1}\right)  _{,1}\right]
\label{24}%
\end{equation}
where $P_{0}$ is the CR Paneitz \ operator (see section$\ 2$). This suggests
us there is an obstruction to the existence of pseudo-Einstein contact form
pertaining to the CR Paneitz operator. See Theorem \ref{t2} below for more details.
\end{remark}

As for the proof of the case of vanishing pseudohermitian torsion :

\begin{proof}
(Proof of Corollary \ref{c2})

Setting $A_{11}=0$ in $\left(  \ref{13}\right)  $, by Theorem \ref{t1}, it
suffices to show that
\[
\gamma_{1,0}=0
\]
in order to have a globally defined pseudo-Einstein contact form
$\widetilde{\theta}=e^{\frac{\left(  f+2u\right)  }{3}}\theta$.

Note that, from $\left(  \ref{0}\right)  $ and $A_{11}=0$,
\begin{equation}
R_{,1}=2u_{\overline{1}11}-i\gamma_{1,0}.\label{8}%
\end{equation}

Utilizing integration by parts, it follows from (\ref{8}) and $\gamma
_{\overline{1},1}=0$ that%

\[%
\begin{array}
[c]{ccl}%
0 & \leq & \int_{M}\left \vert \gamma_{1,0}\right \vert ^{2}d\mu \\
& = & -\int_{M}\gamma_{1}\gamma_{\overline{1},00}d\mu \\
& = & -\int_{M}\gamma_{1}\left(  -iR_{,\overline{1}}+2iu_{,1\overline
{1}\overline{1}}\right)  _{,0}d\mu \  \\
& = & i\int_{M}\gamma_{1}\left(  R_{,0}-2u_{,1\overline{1}0}\right)
_{,\overline{1}}d\mu \\
& = & -i\int_{M}\gamma_{1,\overline{1}}\left(  R_{,0}-2u_{,1\overline{1}%
0}\right)  d\mu \\
& = & 0.
\end{array}
.
\]

The third equality comes from $\left(  \ref{6}\right)  $ and $A_{11}=0$. Then
\[
\gamma_{1,0}=0.
\]

\end{proof}

Before giving the proof of Theorem \ref{t2}, we need the following
B\^{o}chner-type equality.

\begin{lemma}
\label{l8} Let $(M,J,\theta)$ be a closed strictly pseudoconvex CR
$3$-manifold and $\widetilde{\theta}=e^{\frac{\left(  f+2u\right)  }{3}}%
\theta$ is a pseudo-Einstein contact form. Then we have%
\begin{equation}%
%TCIMACRO{\dint \limits_{M}}%
%BeginExpansion
{\displaystyle \int \limits_{M}}
%EndExpansion
\left(  2R-Tor\right)  \left(  \gamma,\gamma \right)  d\mu+2%
%TCIMACRO{\dint \limits_{M}}%
%BeginExpansion
{\displaystyle \int \limits_{M}}
%EndExpansion
\left \vert \gamma_{1,1}\right \vert ^{2}d\mu+\frac{1}{2}%
%TCIMACRO{\dint \limits_{M}}%
%BeginExpansion
{\displaystyle \int \limits_{M}}
%EndExpansion
\left(  P_{0}f\right)  fd\mu=0.\label{26A}%
\end{equation}

\end{lemma}

\begin{proof}
From Theorem \ref{t1} and the commutation formula, it follows that%
\[
\widetilde{\theta}=e^{\frac{\left(  f+2u\right)  }{3}}\theta
\]
is a pseudo-Einstein contact form if and only if
\begin{equation}
P_{1}f=iA_{11}\gamma_{\overline{1}}+R\gamma_{1}-\gamma_{1,1\overline{1}%
}.\label{27A}%
\end{equation}

By the fact that $\gamma_{1,\overline{1}}=0$, it's easy to see
\[
\int_{M}\left(  P_{1}f\right)  \gamma_{\overline{1}}d\mu=\int_{M}\left(
f_{\overline{1}11}+iA_{11}f_{\overline{1}}\right)  \gamma_{\overline{1}}%
d\mu=i\int_{M}A_{11}f_{\overline{1}}\gamma_{\overline{1}}d\mu.
\]

Then, substituting $\left(  \ref{27A}\right)  $ into the last equality and
adding its conjugation, we have%
\begin{equation}
-\int_{M}Tor\left(  d_{b}f,\gamma \right)  d\mu=\int_{M}\left(  2R-Tor\right)
\left(  \gamma,\gamma \right)  d\mu+2\int_{M}\left \vert \gamma_{1,1}\right \vert
^{2}d\mu.\label{28A}%
\end{equation}

On the other hand, the equality $\left(  \ref{27A}\right)  $ and the
commutation formulas enable us to get%
\[%
\begin{array}
[c]{ccl}%
\int_{M}\left(  P_{1}f\right)  f_{\overline{1}}d\mu & = & \int_{M}\left(
iA_{11}\gamma_{\overline{1}}+R\gamma_{1}\right)  f_{\overline{1}}d\mu-\int
_{M}\gamma_{1}f_{\overline{1}\overline{1}1}d\mu \\
& = & \int_{M}\left(  iA_{11}\gamma_{\overline{1}}+R\gamma_{1}\right)
f_{\overline{1}}d\mu \\
&  & +\int_{M}\gamma_{1}\left(  -f_{\overline{1}1\overline{1}}+if_{\overline
{1}0}-R\gamma_{1}f_{\overline{1}}\right)  d\mu \\
& = & \int_{M}iA_{11}\gamma_{\overline{1}}f_{\overline{1}}d\mu+\int_{M}%
\gamma_{1}\left(  -f_{\overline{1}1\overline{1}}+if_{\overline{1}0}\right)
d\mu \\
& = & \int_{M}iA_{11}\gamma_{\overline{1}}f_{\overline{1}}d\mu+\int_{M}%
\gamma_{1}\left(  -f_{\overline{1}1\overline{1}}+if_{0\overline{1}%
}-iA_{\overline{1}\overline{1}}f_{1}\right)  d\mu \\
& = & i\int_{M}\left(  A_{11}\gamma_{\overline{1}}f_{\overline{1}%
}-A_{\overline{1}\overline{1}}\gamma_{1}f_{1}\right)  d\mu \\
& = & -\int_{M}Tor\left(  d_{b}f,\gamma \right)  d\mu.
\end{array}
\]

By the definition of \ the CR Paneitz operator, we obtain
\begin{equation}
\int_{M}\left(  P_{0}f\right)  fd\mu=-\int_{M}\left(  \left(  P_{1}f\right)
f_{\overline{1}}+\left(  P_{\overline{1}}f\right)  f_{1}\right)  d\mu
=2\int_{M}Tor\left(  d_{b}f,\gamma \right)  d\mu \label{29A}%
\end{equation}

Therefore, it follows from the equalities $\left(  \ref{28A}\right)  $ and
$\left(  \ref{29A}\right)  $ that
\[
\int_{M}\left(  2R-Tor\right)  \left(  \gamma,\gamma \right)  d\mu+2\int
_{M}\left \vert \gamma_{1,1}\right \vert ^{2}d\mu+\frac{1}{2}\int_{M}\left(
P_{0}f\right)  fd\mu=0.
\]

Then we are done.
\end{proof}

\bigskip

Such equality enables us to prove Theorem \ref{t2} as follows:

\begin{proof}
(Proof of Theorem \ref{t2}) From the equality $\left(  \ref{26A}\right)  $ and
the hypotheses, it is clear that if $\widetilde{\theta}=e^{\frac{\left(
f+2u\right)  }{3}}\theta$ is a pseudo-Einstein contact form, then%
\[
\gamma=0.
\]
Hence we can solve the inhomogeneous tangential Cauchy-Riemann equation
\[
\overline{\partial}_{b}\varphi=\sigma_{\overline{1}}\theta^{\overline{1}}%
\]
by Lemma \ref{l3}. Note that this implicitly implies $f$ is CR-pluriharmonic.
So the sufficient part is completed.

As for the necessary part, it's obvious from Theorem \ref{t1}.
\end{proof}

Before to go further, we need the following key lemma.

\begin{lemma}
\label{l5} Let $(M,J,\theta)$ be a closed strictly pseudoconvex CR
$3$-manifold with $c_{1}(T_{1,0}M)=0.$ With the notations as above, the
following equality holds%
\begin{equation}
\int_{M}\left(  R-\frac{1}{2}Tor-\frac{1}{2}Tor^{\prime}\right)  \left(
\gamma,\gamma \right)  d\mu+\int_{M}\left \vert \gamma_{1,1}\right \vert ^{2}%
d\mu+\int_{M}Qud\mu+\int_{M}\left(  P_{0}u^{\bot}\right)  u^{\bot}%
d\mu=0.\label{2018BB}%
\end{equation}

\end{lemma}

\begin{proof}
From the equality $\left(  \ref{0}\right)  $, we are able to get%
\[%
\begin{array}
[c]{ccl}%
\left(  R_{,1}-iA_{11,\overline{1}}\right)  \gamma_{\overline{1}} & = &
W_{1}\gamma_{\overline{1}}\\
& = & 2\left(  u_{\overline{1}11}+iA_{11}u_{\overline{1}}\right)
\gamma_{\overline{1}}+iA_{11}\gamma_{\overline{1}}\gamma_{\overline{1}%
}-i\gamma_{1,0}\gamma_{\overline{1}}\\
& = & 2\left(  u_{\overline{1}11}+iA_{11}u_{\overline{1}}\right)
\gamma_{\overline{1}}+iA_{11}\gamma_{\overline{1}}\gamma_{\overline{1}%
}-\left(  \gamma_{1,1\overline{1}}-R\gamma_{1}\right)  \gamma_{\overline{1}}.
\end{array}
\]

Taking the integration over $M$ of both sides and its conjugation, we have, by
the fact that $\gamma_{1,\overline{1}}=0$,
\[
i\int_{M}\left(  A_{11,\overline{1}}\gamma_{\overline{1}}-A_{\overline
{1}\overline{1},1}\gamma_{1}\right)  d\mu+\int_{M}\left(  2R-Tor\right)
\left(  \gamma,\gamma \right)  d\mu+2\int_{M}\left \vert \gamma_{1,1}\right \vert
^{2}d\mu-2\int_{M}Tor\left(  d_{b}u,\gamma \right)  d\mu=0.
\]
That is
\begin{equation}
\int_{M}\left(  R-\frac{1}{2}Tor-\frac{1}{2}Tor^{\prime}\right)  \left(
\gamma,\gamma \right)  d\mu+\int_{M}\left \vert \gamma_{1,1}\right \vert ^{2}%
d\mu-\int_{M}Tor\left(  d_{b}u,\gamma \right)  d\mu=0.\label{34}%
\end{equation}
On the other hand, it follows from the equality $\left(  \ref{0}\right)  $
that
\begin{equation}
\left(  R,_{1}-iA_{11},_{\overline{1}}\right)  u_{\overline{1}}=W_{1}%
u_{\overline{1}}=\left[  2P_{1}u+i\left(  A_{11}\gamma_{\overline{1}}%
-\gamma_{1,0}\right)  \right]  u_{\overline{1}}.\label{30A}%
\end{equation}
By the fact that $\gamma_{1,\overline{1}}=0$, we see that%
\begin{equation}%
\begin{array}
[c]{ccl}%
\int_{M}\gamma_{1,0}u_{\overline{1}}d\mu & = & \int_{M}\gamma_{1}%
u_{\overline{1}0}d\mu \\
& = & -\int_{M}\gamma_{1}\left(  u_{0\overline{1}}-A_{\overline{1}\overline
{1}}u_{1}\right)  d\mu \\
& = & \int_{M}A_{\overline{1}\overline{1}}u_{1}\gamma_{1}d\mu.
\end{array}
\label{31A}%
\end{equation}
It follows from $\left(  \ref{30A}\right)  \ $and$\  \left(  \ref{31A}\right)
$ that%
\[%
\begin{array}
[c]{l}%
\  \  \ 2\int_{M}Qud\mu+2\int_{M}\left(  P_{0}u\right)  ud\mu \\
=i\int_{M}\left[  \left(  A_{11}u_{\overline{1}}\gamma_{\overline{1}%
}-A_{\overline{1}\overline{1}}u_{1}\gamma_{1}\right)  -conj\right]  d\mu \\
=-2\int_{M}Tor\left(  d_{b}u,\gamma \right)  d\mu.
\end{array}
\]
Thus by (\ref{34})
\[
\int_{M}\left(  R-\frac{1}{2}Tor-\frac{1}{2}Tor^{\prime}\right)  \left(
\gamma,\gamma \right)  d\mu+\int_{M}\left \vert \gamma_{1,1}\right \vert ^{2}%
d\mu+\int_{M}Qud\mu+\int_{M}\left(  P_{0}u^{\bot}\right)  u^{\bot}d\mu=0.
\]

\end{proof}

\bigskip

\begin{proof}
(proof of Theorem \ref{t3} and Corollary \ref{c4}) \ If we assume that
\[
\ker P_{1}=\ker P_{0}.
\]
Then we also have
\begin{equation}
0=\int_{M}\left(  R-\frac{1}{2}Tor-\frac{1}{2}Tor^{\prime}\right)  \left(
\gamma,\gamma \right)  d\mu+\int_{M}\left \vert \gamma_{1,1}\right \vert ^{2}%
d\mu+\int_{M}Qu^{\bot}d\mu+\int_{M}\left(  P_{0}u^{\bot}\right)  u^{\bot}%
d\mu.\label{2018F}%
\end{equation}
Here we have used the fact that $P_{0}$ is self-adjoint and%
\[
\int_{M}Tor\left(  d_{b}u_{\ker},\gamma \right)  d\mu=0.
\]

Now if $\widetilde{\theta}=e^{\frac{\left(  f+2u\right)  }{3}}\theta$ is a
pseudo-Einstein contact form for any CR-pluriharmonic function $f,$ it follows
from (\ref{26A}) that
\[
\gamma=0
\]
and then
\[
0=\int_{M}Qu^{\bot}d\mu+\int_{M}\left(  P_{0}u^{\bot}\right)  u^{\bot}%
d\mu=\int_{M}Q^{\bot}u^{\bot}d\mu+\int_{M}\left(  P_{0}u^{\bot}\right)
u^{\bot}d\mu.
\]
By \ the H\H{o}rder inequality and essentially positivity of the CR Paneitz
operator, we have%
\[%
\begin{array}
[c]{ccl}%
\int_{M}Q^{\bot}u^{\bot}d\mu+\int_{M}\left(  P_{0}u^{\bot}\right)  u^{\bot
}d\mu & \geq & \Lambda \int_{M}(u^{\bot})^{2}d\mu-(\int_{M}(Q^{\bot})^{2}%
d\mu)^{\frac{1}{2}})(\int_{M}(u^{\bot})^{2}d\mu)^{\frac{1}{2}}\\
& \geq & [\Lambda(\int_{M}(u^{\bot})^{2}d\mu)^{\frac{1}{2}}-(\int_{M}(Q^{\bot
})^{2}d\mu)^{\frac{1}{2}}](\int_{M}(u^{\bot})^{2}d\mu)^{\frac{1}{2}}\text{ }%
\end{array}
\]
and then
\[
0\geq \Lambda(\int_{M}(u^{\bot})^{2}d\mu)^{\frac{1}{2}}-(\int_{M}(Q^{\bot}%
)^{2}d\mu)^{\frac{1}{2}}.
\]
Hence
\[
\int_{M}(Q^{\bot})^{2}d\mu \geq \Lambda^{2}\int_{M}(u^{\bot})^{2}d\mu.
\]

Furthermore, if he CR \ $Q$-curvature is pluriharmonic (i.e. $Q^{\bot}=0),$
then
\[
u^{\bot}=0
\]
and by (\ref{0})
\[
W_{1}=0.
\]
Hence $\theta$ is also a globally defined pseudo-Einstein contact form.
Moreover, if the torsion is vanishing, then $(M,J,\theta)$ is the Sasakian
space form.
\end{proof}

\bigskip

\begin{proof}
(proof of Theorem \ref{t3a} and Corollary \ref{c5}) As before
\[%
\begin{array}
[c]{ccl}%
\int_{M}Qu^{\bot}d\mu+\int_{M}\left(  P_{0}u^{\bot}\right)  u^{\bot}d\mu & = &
\int_{M}Q^{\bot}u^{\bot}d\mu+\int_{M}\left(  P_{0}u^{\bot}\right)  u^{\bot
}d\mu \\
& \geq & \Lambda \int_{M}(u^{\bot})^{2}d\mu \text{ }\\
& \geq & 0
\end{array}
\]
if
\begin{equation}
Q^{\bot}=0.\label{2018G}%
\end{equation}
It follows from (\ref{2018F}) that
\begin{equation}
0\geq \int_{M}\left(  R-\frac{1}{2}Tor-\frac{1}{2}Tor^{\prime}\right)  \left(
\gamma,\gamma \right)  d\mu+\int_{M}\left \vert \gamma_{1,1}\right \vert ^{2}%
d\mu.\label{2018BBB}%
\end{equation}
if (\ref{2018G}) holds. Hence%
\[
\gamma=0\text{ }%
\]
if the pseudohermitian curavture is $\left(  \frac{1}{2},\frac{1}{2}\right)
$\textit{-positive.} It follows from Theorem \ref{t2} that $M$ admits a
globally defined pseudo-Einstein contact form $\widetilde{\theta}%
=e^{\frac{\left(  f+2u\right)  }{3}}\theta$. \ 

Furthermore, if he CR \ $Q$-curvature is pluriharmonic (i.e. $Q^{\bot}=0),$
then
\[
u^{\bot}=0
\]
and by (\ref{0})
\[
W_{1}=0.
\]
Hence $\theta$ is also a globally defined pseudo-Einstein contact form.

Now if $(M,J,\theta)$ is spherical and pseudo-Einstein, we have
\[
W_{1}=R,_{1}-iA_{11,\overline{1}}=0
\]
and
\[
iR_{,11}=3RA_{11}+6iA_{11,0}-4A_{11,\overline{1}1}.
\]

By cancelling $R_{,11}$, one derives%
\[
3RA_{11}+6iA_{11,0}-3A_{11,\overline{1}1}=0.
\]

On the other hand, it follows from the commutation relation (\cite{l}) that%
\[
A_{11,1\overline{1}}-A_{11,\overline{1}1}=iA_{11,0}+2RA_{11},
\]

we obtain%
\[
-3RA_{11}+2A_{11,1\overline{1}}-3A_{11,\overline{1}1}=0
\]
and then%
\[
-2%
%TCIMACRO{\dint _{M}}%
%BeginExpansion
{\displaystyle \int_{M}}
%EndExpansion
\left \vert A_{11,1}\right \vert ^{2}d\mu+3%
%TCIMACRO{\dint _{M}}%
%BeginExpansion
{\displaystyle \int_{M}}
%EndExpansion
\left \vert A_{11,\overline{1}}\right \vert ^{2}d\mu=3%
%TCIMACRO{\dint _{M}}%
%BeginExpansion
{\displaystyle \int_{M}}
%EndExpansion
R\left \vert A_{11}\right \vert ^{2}d\mu.
\]
Moreover, if $R$ is positive constant, then
\[
\text{\ }A_{11,\overline{1}}=0
\]
and \
\[
-2%
%TCIMACRO{\dint _{M}}%
%BeginExpansion
{\displaystyle \int_{M}}
%EndExpansion
\left \vert A_{11,1}\right \vert ^{2}d\mu=3%
%TCIMACRO{\dint }%
%BeginExpansion
{\displaystyle \int}
%EndExpansion
R\left \vert A_{11}\right \vert ^{2}%
\]
which implies
\[
A_{11}=0.
\]
It follows that $(M,J,\theta)$ is the Sasakian space form with the positive
Tanaka-Webster constant scalar curvature and vanishing torsion.
\end{proof}

\bigskip

\begin{proof}
(proof of Theorem \ref{t4} and Corollary \ref{c6})

It follows from (\ref{2018BB}) that
\[
\int_{M}\left(  R-\frac{1}{2}Tor-\frac{1}{2}Tor^{\prime}\right)  \left(
\gamma,\gamma \right)  d\mu+\int_{M}\left \vert \gamma_{1,1}\right \vert ^{2}%
d\mu+\int_{M}\left(  P_{0}u^{\bot}\right)  u^{\bot}d\mu=0
\]
if
\[
Q=0.
\]
Then we are again in the line of (\ref{2018BBB}) and we have%
\[
\gamma=0
\]
under the assumptions. Again we have
\[
u^{\bot}=0
\]
and by (\ref{0})
\[
R_{,1}-iA_{11},_{\bar{1}}=2P_{1}u_{\ker}.
\]
Hence
\[
\widetilde{\theta}=e^{\frac{\left(  f+2u\right)  }{3}}\theta=e^{\frac{\left(
f+2u_{\ker}\right)  }{3}}\theta
\]
which is pseudo-Einstein. It follows from (\ref{0c}) that
\[
P_{1}u_{\ker}=0
\]
and then
\[
R_{,1}-iA_{11},_{\bar{1}}=0
\]
which is pseudo-Einstein as well. The argumet for CR uniformization theorem as
in Corollary \ref{c6} are easily derived from the previous one. Then we are done.
\end{proof}

\bigskip

\section{Appendix}

In this appendix, we introduce some basic notions from pseudohermitian
geometry as in \cite{l}.

\begin{definition}
\label{def:CR} Let $M$ be a smooth manifold and $\xi \subset TM$ a subbundle. A
\textbf{CR structure} on $\xi$ consists of an endomorphism $J:\xi
\rightarrow \xi$ with $J^{2}=-id$ such that the following integrability
condition holds.

1. If $X,Y\in \xi$, then so is $[JX,Y]+[X,JY]$.

2. $J([JX,Y]+[X,JY])=[JX,JY]-[X,Y]$.
\end{definition}

The CR structure $J$ can be extended to $\xi \otimes{\mathbb{C}}$, which we can
then decompose into the direct sum of eigenspaces of $J$. The eigenvalues of
$J$ are $i$ and $-i$, and the corresponding eigenspaces will be denoted by
$T^{1,0}$ and $T^{0,1}$, respectively. The integrability condition can then be
reformulated as
\[
X,Y\in T^{1,0} \text{ implies } [X,Y]\in T^{1,0}.
\]

Now consider a closed $2n+1$-manifold $M$ with a cooriented contact structure
$\xi=\ker \theta$. This means that $\theta \wedge d\theta^{n} \neq0 $. The
\textbf{Reeb vector field} of $\theta$ is the vector field $T$ uniquely
determined by the equations
\begin{equation}
{\theta}(T)=1,\quad \text{and}\quad d{\theta}(T,{\cdot}%
)=0.\label{eq:defining_Reeb}%
\end{equation}

A \textbf{pseudohermitian manifold} is a triple $(M^{2n+1},\theta,J),$ where
$\theta$ is a contact form on $M$ and $J$ is a CR structure on $\ker \theta$.
The \textbf{Levi form} $\left \langle \ ,\  \right \rangle $ is the Hermitian
form on $T^{1,0}$ defined by
\[
H(Z,W)=\left \langle Z,W\right \rangle =-i\left \langle d\theta,Z\wedge
\overline{W}\right \rangle .
\]

We can extend this Hermitian form $\left \langle \ ,\  \right \rangle $ to
$T^{0,1}$ by defining $\left \langle \overline{Z},\overline{W}\right \rangle
=\overline{\left \langle Z,W\right \rangle }$ for all $Z,W\in T^{1,0}$.
Furthermore, the Levi form naturally induces a Hermitian form on the dual
bundle of $T^{1,0}$, and hence on all induced tensor bundles.

We now restrict ourselves to \textbf{strictly pseudoconvex} CR manifolds, or
in other words compatible complex structures $J$. This means that the Levi
form induces a Hermitian metric $\langle \cdot,\cdot \rangle_{J,{\theta}}$ by
\[
\langle V,U\rangle_{J,{\theta}}=d{\theta}(V,JU).
\]
The associated norm is defined as usual: $|V|_{J,\theta}^{2}=\langle
V,V\rangle_{J,{\theta}}$. It follows that $H$ also gives rise to a Hermitian
metric for $T^{1,0}$, and hence we obtain Hermitian metrics on all induced
tensor bundles. By integrating this Hermitian metric over $M$ with respect to
the volume form $d\mu=\theta \wedge d\theta^{n}$, we get an $L^{2}$-inner
product on the space of sections of each tensor bundle.

The \textbf{pseudohermitian connection} or \textbf{Tanaka-Webster connection}
(\cite{ta}, \cite{we}) of $(J,\theta)$ is the connection $\nabla$ on
$TM\otimes \mathbb{C}$ (and extended to tensors) given in terms of a local
frame $\{Z_{\alpha}\}$ for $T^{1,0}$ by%

\[
\nabla Z_{\alpha}=\omega_{\alpha}{}^{\beta}\otimes Z_{\beta}, \quad \nabla
Z_{\bar{ \alpha}}=\omega_{\bar{\alpha}}{}^{\bar{\beta}}\otimes Z_{\bar{\beta}%
},\quad \nabla T=0,
\]
where $\omega_{\alpha}{}^{\beta}$ is the $1$-form uniquely determined by the
following equations:%

\begin{equation}%
\begin{split}
d\theta^{\beta} &  =\theta^{\alpha}\wedge \omega_{\alpha}{}^{\beta}%
+\theta \wedge \tau^{\beta}\\
\tau_{\alpha}\wedge \theta^{\alpha} &  =0\\
\omega_{\alpha}{}^{\beta}+\omega_{\bar{\beta}}{}^{\bar{\alpha}} &  =0.
\end{split}
\end{equation}
Here $\tau^{\alpha}$ is called the \textbf{pseudohermitian torsion}, which we
can also write as
\[
\tau_{\alpha}=A_{\alpha \beta}\theta^{\beta}.
\]
The components $A_{\alpha \beta}$ satisfy
\[
A_{\alpha \beta}=A_{\beta \alpha}.
\]
We often consider the \textbf{torsion tensor} given by
\[
A_{J,\theta}=A^{\alpha}{}_{\bar{\beta}}Z_{\alpha}\otimes \theta^{\bar{\beta}%
}+A^{\bar{\alpha}}{}_{\beta}Z_{\bar{\alpha}}\otimes \theta^{\beta}.
\]

We now consider the curvature of the Tanaka-Webster connection in terms of the
coframe $\{ \theta=\theta^{0},\theta^{\alpha},\theta^{\bar{\beta}}\} $. The
second structure equation gives
\[%
\begin{split}
\Omega_{\beta}{}^{\alpha}  &  =\overline{\Omega_{\bar{\beta}}{}^{\bar{\alpha}%
}}=d\omega_{\beta}{}^{\alpha}-\omega_{\beta}{}^{\gamma}\wedge \omega_{\gamma}%
{}^{\alpha},\\
\Omega_{0}{}^{\alpha}  &  =\Omega_{\alpha}{}^{0}=\Omega_{0}{}^{\bar{\beta}%
}=\Omega_{\bar{\beta}}{}^{0}=\Omega_{0}{}^{0}=0.
\end{split}
\]

In \cite[Formulas 1.33 and 1.35]{we}, Webster showed that the curvature
$\Omega_{\beta}{}^{\alpha}$ can be written as
\begin{equation}%
\begin{array}
[c]{c}%
\Omega_{\beta}{}^{\alpha}=R_{\beta}{}^{\alpha}{}_{\rho \bar{\sigma}}%
\theta^{\rho}\wedge \theta^{\bar{\sigma}}+W_{\beta}{}^{\alpha}{}_{\rho}%
\theta^{\rho}\wedge \theta-W^{\alpha}{}_{\beta \bar{\rho}}\theta^{\bar{\rho}%
}\wedge \theta+i\theta_{\beta}\wedge \tau^{\alpha}-i\tau_{\beta}\wedge
\theta^{\alpha},
\end{array}
\label{a}%
\end{equation}
where the coefficients satisfy
\[%
\begin{array}
[c]{c}%
R_{\beta \bar{\alpha}\rho \bar{\sigma}}=\overline{R_{\alpha \bar{\beta}\sigma
\bar{\rho}}}=R_{\bar{\alpha}\beta \bar{\sigma}\rho}=R_{\rho \bar{\alpha}%
\beta \bar{\sigma}},\  \ W_{\beta \bar{\alpha}\gamma}=W_{\gamma \bar{\alpha}\beta
}.
\end{array}
\]
In addition, by \cite[(2.4)]{l} the coefficients $W_{\alpha}{}^{\beta}{}%
_{\rho}$ are determined by the torsion,
\[
W_{\alpha}{}^{\beta}{}_{\rho}=A_{\alpha \rho,}{}^{\beta}.
\]
Contraction of (\ref{a}) yields%
\begin{equation}%
\begin{split}
\Omega_{\alpha}{}^{\alpha}=d\omega_{\alpha}{}^{\alpha} &  =R_{\rho \bar{\sigma
}}\theta^{\rho}\wedge \theta^{\bar{\sigma}}+W_{\alpha}{}^{\alpha}{}_{\rho
}\theta^{\rho}\wedge \theta-W_{\overline{\alpha}}{}^{\overline{\alpha}}{}%
_{\bar{\rho}}\theta^{\bar{\rho}}\wedge \theta \\
&  =R_{\rho \bar{\sigma}}\theta^{\rho}\wedge \theta^{\bar{\sigma}}+A_{\alpha
\rho}{}^{\alpha}\theta^{\rho}\wedge \theta-A_{\bar{\alpha}\bar{\rho}}{}%
^{\bar{\alpha}}\theta^{\bar{\rho}}\wedge \theta
\end{split}
\label{b}%
\end{equation}

We will denote components of covariant derivatives by indices preceded by a
comma. For instance, we write $A_{\alpha \beta,\gamma}$. Here the indices
$\{0,\alpha,\bar{\beta}\}$ indicate derivatives with respect to $\{T,Z_{\alpha
},Z_{\bar{\beta}}\}$. For derivatives of a scalar function, we will often omit
the comma. For example, $\varphi_{\alpha}=Z_{\alpha}\varphi,\  \varphi
_{\alpha \bar{\beta}}=Z_{\bar{\beta}}Z_{\alpha}\varphi-\omega_{\alpha}%
{}^{\gamma}(Z_{\bar{\beta}})Z_{\gamma}\varphi,\  \varphi_{0}=T \varphi$ for a
(smooth) function $\varphi$.

In particular, we define followings for $n=1$ For a real function $\varphi$,
the subgradient $\nabla_{b}$ is defined by $\nabla_{b}\varphi \in \xi$ and
$\left \langle Z,\nabla_{b}\varphi \right \rangle _{L_{\theta}}=d\varphi(Z)$ for
all vector fields $Z$ tangent to contact plane. Locally $\nabla_{b}%
\varphi=\varphi_{\bar{1}}Z_{1}+\varphi_{1}Z_{\bar{1}}$. We can use the
connection to define the subhessian as the complex linear map%

\[
(\nabla^{H})^{2}\varphi:T_{1,0}\oplus T_{0,1}\rightarrow T_{1,0}\oplus
T_{0,1},
\]
by
\[
(\nabla^{H})^{2}\varphi(Z)=\nabla_{Z}\nabla_{b}\varphi.\  \
\]
Also
\[
\Delta_{b}\varphi=Tr\left(  (\nabla^{H})^{2}\varphi \right)  =(\varphi
_{1\bar{1}}+\varphi_{\bar{1}1}).
\]
For all $Z=x^{1}Z_{1}\in T_{1,0}$, we define%

\[%
\begin{split}
Ric(Z,Z) &  =Wx^{1}x^{\bar{1}}=W|Z|_{L_{\theta}}^{2},\\
Tor(Z,Z) &  =2Re\ iA_{\bar{1}\bar{1}}x^{\bar{1}}x^{\bar{1}}.
\end{split}
\]

We also need the following commutation relations (\cite{l}).%
\begin{equation}%
\begin{array}
[c]{ccl}%
C_{I,01}-C_{I,10} & = & C_{I,\overline{1}}A_{11}-kC_{I,}A_{11,\overline{1}},\\
C_{I,0\overline{1}}-C_{I,\overline{1}0} & = & C_{I,1}A_{\overline{1}%
\overline{1}}-kC_{I,}A_{\overline{1}\overline{1},1},\\
C_{I,1\overline{1}}-C_{I,\overline{1}1} & = & iC_{I,0}+kWC_{I}.
\end{array}
\label{6}%
\end{equation}
Here $C_{I}$ denotes a coefficient of a tensor with multi-index $I$ consisting
of only $1$ and $\bar{1}$, and $k$ is the number of $1$'s minus the number of
$\bar{1}$'s in $I$.

\end{document}